\documentclass[12pt]{extarticle}
\usepackage{amsmath, amsthm, amssymb, hyperref, color}
\usepackage[shortlabels]{enumitem}
\usepackage{graphicx}
\usepackage[all]{xypic}
\usepackage{makecell}

\oddsidemargin \evensidemargin
\newtheorem{theorem}{Theorem}
\newtheorem{proposition}[theorem]{Proposition}

\newtheorem{corollary}[theorem]{Corollary}

\newtheorem{problem}[theorem]{Problem}
\newtheorem{remark}[theorem]{Remark}
\newtheorem{example}[theorem]{Example}
\newtheorem{conjecture}[theorem]{Conjecture}

\newcommand{\RR}{\mathbb{R}}

\newcommand{\PP}{\mathbb{P}}

 \date{}
 
\title{\textbf{Moment Varieties of Gaussian Mixtures}}
\author{
Carlos Am\'endola, Jean-Charles Faug\`{e}re, and Bernd Sturmfels}

\begin{document}

\maketitle

\begin{abstract}
\noindent The points of a moment variety
are the vectors of all moments up to some order
of a family of probability distributions.
We study this variety for  mixtures of Gaussians.
Following up on Pearson's classical work from 1894,
we apply current tools from computational algebra
to recover the parameters from the moments.
Our moment varieties extend
objects familiar to algebraic geometers. For
instance, the secant varieties of Veronese varieties
are the loci obtained by setting all covariance matrices to zero.
We compute the ideals of the
5-dimensional moment varieties  representing
mixtures of two univariate Gaussians, and we
offer a comparison to the  maximum likelihood approach.
\end{abstract}

\section{Introduction}

The $n$-dimensional Gaussian distribution is defined by the moment generating function
\begin{equation}
\label{eq:gaussian}
\sum_{i_1,i_2,\ldots,i_n \geq 0}
\frac{m_{i_1 i_2 \cdots i_n}}{i_1 ! i_2 ! \cdots i_n !} 
t_1^{i_1} t_2^{i_2} \cdots t_n^{i_n} \,\,\,=\,\,\,
{\rm exp}(t_1 \mu_1 + \cdots + t_n \mu_n) \cdot
{\rm exp} \biggl( \frac{1}{2} \sum_{i,j=1}^n \sigma_{ij} t_i t_j \biggr).
\end{equation}
The model parameters are the entries of the mean
$\mu = (\mu_1,\ldots,\mu_n)$
and  of the covariance matrix
$\,\Sigma = (\sigma_{ij})$.
The unknowns $\mu_i$ have degree $1$,
and the unknowns $\sigma_{ij}$ have degree~$2$.
Then $m_{i_1 i_2 \cdots i_n}$ is a homogeneous polynomial 
of degree $i_1{+}i_2{+}\cdots{+}i_n$ in the $n+\binom{n+1}{2}$  unknowns.

Let $\mathbb{P}^N$ be the projective space of dimension 
$N = \binom{n+d}{d}-1$ whose coordinates
are all $N+1$ moments $m_{m_{i_1 i_2 \cdots i_n}}$
with $i_1 + i_2 + \cdots + i_n \leq d$. The
closure of the image of parametrization (\ref{eq:gaussian})
is a subvariety $\mathcal{G}_{n,d} $ of $\PP^N$.
We called this the
{\em Gaussian moment variety of order $d$}.
Its dimension equals $n + \binom{n+1}{2}$.
 In Section~\ref{sec:two} we discuss this variety and its
 defining polynomials.
 
The main object of study in this paper is the
secant variety $\sigma_k (\mathcal{G}_{n,d})$ 
of the Gaussian moment variety.
That variety is the Zariski closure of the set of vectors
of moments of order $\leq d$ of any distribution on $\RR^n$
that is the mixture of $k$ Gaussians, for $k=2,3,\ldots$.
In short,  $\sigma_k (\mathcal{G}_{n,d})$ is the projective variety
that represents mixtures of $k$ Gaussians.
 Since such mixtures are identifiable \cite{Yak}, this secant variety eventually has
the expected dimension:
\begin{equation}
\label{eq:mixturedim}
 {\rm dim} (\sigma_k (\mathcal{G}_{n,d}))
\,= \,  k \cdot \biggl[n + \binom{n+1}{2} \biggr]  + k-1 \qquad \hbox{for} \,\, d \gg 0 . 
\end{equation}
The parametrization of $\sigma_k(\mathcal{G}_{n,d})$ is given
by replacing the right hand side of (\ref{eq:gaussian}) with
a convex combination of $k$ such expressions.
The number of model parameters is the right hand side of (\ref{eq:mixturedim}).
If the moments $m_{i_1 i_2 \cdots i_n}$ are derived numerically from data,
then one obtains a system of polynomial equations whose
unknowns are the  model parameters.  The process
of solving these polynomial equations is known as the {\em method of moments}
for Gaussian mixtures.

For a concrete example,  consider the case $n=1$ and $d=6$.
The Gaussian moment variety $\mathcal{G}_{1,6}$ is a surface 
 of degree $15$ in $\mathbb{P}^6$  that is cut out by $20$ cubics.
 These cubics will be explained in Section~\ref{sec:two}.
 For $k=2$ we obtain the variety of secant lines, here denoted
 $\sigma_2(\mathcal{G}_{1,6})$. This represents mixtures of
 two univariate Gaussians. It has the parametric representation
\begin{equation}
\label{eq:karlpara}
 \begin{matrix}
m_0 & = & \!\!\!\!\!\!\!\!\! \!\!\!\! 1 \\
m_1 & = & \lambda \mu + (1-\lambda) \nu \\
m_2 & = & \lambda (\mu^2 + \sigma^2) + (1-\lambda) (\nu^2 + \tau^2) \\
m_3 & = & \lambda (\mu^3 + 3 \mu \sigma^2) + (1-\lambda) (\nu^3 + 3 \nu \tau^2) \\
m_4 & = & \lambda (\mu^4 + 6 \mu^2 \sigma^2 + 3 \sigma^4) 
          + (1-\lambda) (\nu^4 + 6 \nu^2 \tau^2  + 3 \tau^4) \\
m_5 & = & \lambda (\mu^5 + 10 \mu^3 \sigma^2 + 15 \mu \sigma^4) 
          + (1-\lambda) (\nu^5 + 10 \nu^3 \tau^2  + 15 \nu \tau^4) \\
m_6 & = & \lambda (\mu^6 + 15 \mu^4 \sigma^2 + 45 \mu^2 \sigma^4 + 15 \sigma^6) 
          + (1-\lambda) (\nu^6 + 15 \nu^4 \tau^2  + 45 \nu^2 \tau^4 + 15 \tau^6) 
\end{matrix}
\end{equation}
These are obtained from the first seven coefficients in the moment generating function
$$
\sum_{i=0}^\infty  \frac{m_i }{i!} t^i \,\, = \,\, \lambda 
\cdot {\rm exp}(\mu t + \frac{1}{2} \sigma^2 t^2) 
\,+\, (1-\lambda)  \cdot {\rm exp}( \nu t + \frac{1}{2} \tau^2 t^2).
$$
Here and throughout we
use the standard notation $\sigma^2$ for the variance $\sigma_{11}$
when $n=1$.
The variety $\sigma_2(\mathcal{G}_{1,6})$ is five-dimensional, so it is a
 hypersurface in $\mathbb{P}^6$. In Section~\ref{sec:three} we derive:

\begin{theorem}
\label{thm:thirtynine}
The defining polynomial of $\sigma_2(\mathcal{G}_{1,6})$
is a sum of $\,31154$ monomials of degree $39$.
This polynomial has degrees
$25, 33, 32, 23, 17, 12, 9$  in
$m_0,m_1,m_2,m_3,m_4,m_5,m_6$
respectively.
\end{theorem}

We see in particular that $m_6$ can be recovered from
$m_1,m_2,m_3,m_4$ and $m_5$ by solving a
univariate equation of degree $9$. This number is of
special historic interest.
The 1894 paper \cite{Pea} introduced
the method of moments and, in our view, it is
the first paper in
algebraic statistics.
 Pearson analyzed phenotypic data from
two crab populations, and he showed how to
find the five parameters in (\ref{eq:karlpara}) 
by solving an equation of degree $9$ if the 
first five moments are given. The two occurrences
of the number $9$ are equivalent, in light of Lazard's result
\cite{Laz} that the parameters
$\lambda, \mu,\nu, \sigma,\tau$ are rational functions
 in the first six moments
$m_1,\ldots,m_6$.

The hypersurface in $\PP^6$ described in Theorem \ref{thm:thirtynine}
contains a familiar threefold, namely the
determinantal variety $\sigma_2(\nu_6(\PP^1))$
 defined by the $3 \times 3$-minors of   the  $4 \times 4$-Hankel matrix
\begin{equation}
\label{eq:hankel}
 \begin{pmatrix}
m_0 & m_1 & m_2 & m_3 \\
 m_1 & m_2 & m_3 & m_4 \\
  m_2 & m_3 & m_4 & m_5 \\
   m_3 & m_4 & m_5 & m_6
   \end{pmatrix}.
\end{equation}
This can be seen by setting $\sigma = \nu = 0$ in the
parametrization (\ref{eq:karlpara}). Indeed, if the variances
tend to zero then the Gaussian mixture converges to 
a mixture of the point distributions, supported
at the means $\mu$ and $\nu$. The first $d+1$ moments of
point distributions form the rational normal curve in $\PP^d$,
consisting of  Hankel matrices  of rank $1$.
Their $k$th mixtures specify a secant variety of the rational normal curve,
consisting of Hankel matrices of rank~$k$.

The last four sections of this paper are organized as follows.
In Sections 3 and 4 we focus on
 mixtures of univariate Gaussians. We derive Pearson's hypersurface
    $\sigma_2(\mathcal{G}_{1,6})$ in detail, and we examine the 
    varieties $\sigma_2(\mathcal{G}_{1,d})$ for $d>6$ and 
$\sigma_k(\mathcal{G}_{1,3k})$ for $k=3,4$. 
In Section 5 we apply the
method of moments to the data discussed in \cite[\S 3]{CB},
thus offering a comparison to maximum likelihood estimation.
In Section 6 we explore some cases of the moment varieties for
Gaussian mixtures with $n=2$, and we discuss 
directions for future research.

\section{Gaussian Moment Varieties}
\label{sec:two}

In this section we examine the Gaussian moment varieties
$\mathcal{G}_{n,d}$, starting with the case $n=1$.
The moment variety $\mathcal{G}_{1,d}$ is a surface in $\PP^d$.
Its defining polynomial equations are as follows:

\begin{proposition}
\label{prop:surface}
Let $d \geq 3$.
The homogeneous prime ideal of the Gaussian moment surface $\mathcal{G}_{1,d}$ 
is minimally generated by $\binom{d}{3}$ cubics.
These are the $3 \times 3$-minors of the $3 \times d$-matrix
$$ 
H_d \,\,= \, \left(\begin{array}{ccccccc} 
    0&m_0&2m_1&3m_2 & 4m_3 & \cdots & (d-1) m_{d-2}\\ 
    m_0& m_1& m_2 & m_3 & m_4 & \cdots & m_{d-1}\\
    m_1& m_2& m_3 & m_4 & m_5 &\cdots & m_d\\
    \end{array}\right).
$$  
\end{proposition}

\begin{proof}
Let $I_d = \mathcal{I}(\mathcal{G}_{1,d})$ be the vanishing ideal of the moment surface,
 and let $J_d$ be the ideal generated by the  $3 \times 3$-minors of $H_d$. 
A key observation,
easily checked in (\ref{eq:karlpara}), is that the moments
of the univariate Gaussian distribution
satisfy the  recurrence relation 
\begin{equation}
\label{eq:recurrence}
 m_i \,\,= \,\, \mu ~ m_{i-1} + (i-1) \sigma^2 m_{i-2}
\,\,\, \, \hbox{for} \,\,\, i \geq 1 . 
\end{equation}
Hence the row vector $(\sigma^2, \mu, -1)$ is in the left kernel of $H_d$. 
Thus ${\rm rank}(H_d)=2$, and this means that all 
$3 \times 3$-minors of $H_d$ indeed vanish on the surface
$\mathcal{G}_{1,d}$. This proves $J_d \subseteq I_d$.

From the previous inclusion we have ${\rm dim}(V(J_d)) \geq 2$.
Fix a monomial order such that the antidiagonal product is the leading
term in each of the $3 \times 3$-minors of $H_d$. These leading terms
are the distinct cubic monomials in $m_1,m_2,\ldots,m_{d-2}$. Hence
the initial ideal satisfies
\begin{equation}
\label{eq:initialideal}
\langle m_1,m_2, \ldots,m_{d-2} \rangle^3 \,\, \subseteq \,\, {\rm in}(J_d) .
\end{equation}
This shows that ${\rm dim}(V(J_d)) = {\rm dim}(V({\rm in}(J_d))) \leq 2$,
and hence $V(J_d)$ has dimension 2 in $\PP^d$.

We next argue that $V(J_d)$ is an irreducible surface.
On the affine space $\mathbb{A}^d = \{m_0 = 1\}$, this
clearly holds, even ideal-theoretically,
 because the minor indexed by $1$, $2$ and $i$ 
expresses $m_i$ as a polynomial in $m_1$ and $m_2$.
Consider the intersection of  $V(J_d)$ with $\mathbb{P}^{d-1} = \{m_0 = 0\}$.
The matrix $H_d$ shows that $m_1 = m_2 = \cdots = m_{d-2} = 0$ holds 
on that hyperplane at infinity, so $V(J_d) \cap \{m_0 = 0\}$ is a curve.
Every point on that curve is the limit of points in $V(J_d) \cap \{m_0 = 1\}
= V(I_d) \cap \{m_0 = 1\}$, obtained by making
$(\mu,\sigma)$ larger in an appropriate direction.
This shows that $V(J_d)$ is irreducible, and we conclude
that $V(J_d) = V(I_d)$.

At this point we only need to exclude the possibility that
$J_d$ has lower-dimensional embedded components.
However, there are no such components
  because the
ideal of maximal minors of a $3 \times d$-matrix of unknowns
is Cohen-Macaulay, and $V(J_d)$ has the expected dimension
for an intersection with $\PP^d$. This shows that $J_d$
is a Cohen-Macaulay ideal. Hence $J_d$ has no embedded
associated primes, and we conclude that $J_d  = I_d$ as desired.
\end{proof}

\begin{corollary}
The $3 \times 3$-minors of the matrix $H_d$ form a Gr\"obner basis
for the prime ideal of  the Gaussian moment surface
$\mathcal{G}_{1,d} \subset \PP^d$ with respect to the 
reverse lexicographic term order.
\end{corollary}

\begin{proof}
The ideal $J_d$ of $\mathcal{G}_{1,d}$ is
generated by the  $3 \times 3$-minors of $H_d$.
Our claim states that equality holds in 
(\ref{eq:initialideal}). This can be seen by 
examining the Hilbert series of both ideals.
It is well known that the ideal of $r \times r$-minors
of a generic $r \times d$-matrix has the same
numerator of the Hilbert series as the $r$-th power of
the maximal ideal $\langle m_1,m_2,\ldots,m_{d-r+1} \rangle$.
Since that ideal is Cohen-Macaulay,
this Hilbert series numerator remains unchanged under
transverse linear sections. Hence our ideal $J_d$
has the same Hilbert series numerator as 
$\langle m_1,m_2,\ldots,m_{d-2} \rangle^3$.
This implies that the two ideals in
(\ref{eq:initialideal}) have the same Hilbert series, so they are equal.
\end{proof}

The argument above tells us that our surface has the same degree as
$\langle m_1,m_2,\ldots,m_{d-2} \rangle^3$:

\begin{corollary}
The Gaussian moment surface $\mathcal{G}_{1,d}$ has degree $\binom{d}{2}$ in $\PP^d$.
\end{corollary}

It is natural to ask whether the nice determinantal representation
extends to the varieties $\mathcal{G}_{n,d}$ when $n \geq 2$.
The answer is no, even in the first nontrivial case, when $n=2$ and $d=3$:

\begin{proposition} \label{prop:G23}
The $5$-dimensional variety $\mathcal{G}_{2,3}$ has degree $16$ in $\PP^9$.
Its homogeneous prime ideal is minimally generated by $14$ cubics and $4$ quartics,
and the Hilbert series  equals
$$ \dfrac{1+4t+10t^2+6t^3-4t^4-t^5}{(1-t)^{6}}. $$
Starting from four of the cubics, the ideal can be computed by a 
saturation as follows:
\begin{equation}
\label{eq:saturation}
 \begin{matrix}
\!\!\!\! \langle
\,2m_{10}^3-3 m_{00} m_{10} m_{20} +m_{00}^2 m_{30}\,,\,\, 
2 m_{01} m_{10}^2  - 2 m_{00} m_{10} m_{11}-m_{00} m_{01} m_{20}+m_{00}^2 m_{21}, 
\qquad \qquad \\
\! 2 m_{01}^2 m_{10} {-}m_{00} m_{02} m_{10}{-}2 m_{00} m_{01} m_{11}{+}m_{00}^2 m_{12}, 
2 m_{01}^3 {-} 3 m_{00} m_{01} m_{02}{+}m_{00}^2 m_{03}\,
\rangle  : \langle m_{00} \rangle^\infty .
\end{matrix}
\end{equation}
\end{proposition}

The four special cubics in (\ref{eq:saturation}) above are the 
cumulants $k_{30}, k_{21}, k_{12}, k_{03}$
when expressed in terms of moments.
The same technique works for all $n$ and $d$, and
we shall now explain it.

We next define cumulants.
These form a coordinate system that is more  efficient than the
moments, not just for Gaussians but for any probability distribution on $\RR^n$
that is polynomial in the sense of Belkin and Sinha \cite{BS}.
We introduce two exponential generating functions
$$
M \,\,= \! \sum_{i_1,i_2,\ldots,i_n \geq 0}
\frac{m_{i_1 i_2 \cdots i_n}}{i_1 ! i_2 ! \cdots i_n !} 
t_1^{i_1} t_2^{i_2} \cdots t_n^{i_n} \quad \hbox{and} \quad
K \,\,= \! \sum_{i_1,i_2,\ldots,i_n \geq 0}
\frac{k_{i_1 i_2 \cdots i_n}}{i_1 ! i_2 ! \cdots i_n !} 
t_1^{i_1} t_2^{i_2} \cdots t_n^{i_n} .
$$
Fixing $m_{00 \cdots 0} = 1$ and $k_{00 \cdots 0} = 0$, 
these are related by the  identities of generating functions
\begin{equation}
\label{eq:explog}
 M \,=\,{\rm exp}(K) \quad \hbox{and}
\quad K \,=\, {\rm log}(M). 
\end{equation}
The coefficients are unknowns: the $m_{i_1 i_2 \cdots i_n}$
are  {\em moments}, and  the $k_{i_1 i_2 \ldots i_n}$
are  {\em cumulants}. The  integer
$i_1 + i_2 + \cdots + i_n$ is the {\em order} of the moment 
$m_{i_1 i_2 \cdots i_n}$ or the cumulant
$k_{i_1 i_2 \ldots i_n}$.

The identity (\ref{eq:explog}) expresses
moments of order $\leq d$ as polynomials
in cumulants of order $\leq d$, and vice versa.
Either of these can serve as an affine coordinate system
on the $\PP^N$ whose points are inhomogeneous polynomials
of degree $\leq d$ in $n$ variables.
To be precise, the affine space
$ \mathbb{A}^N  = \{m_{00 \cdots 0} = 1\}$ 
consists of those polynomials whose constant term is nonzero.
Hence the formulas (\ref{eq:explog})
represent a non-linear change of variables on $\mathbb{A}^N$.
This was called {\em Cremona linearization} in \cite{CCMRZ}.
We agree with the authors of \cite{CCMRZ} that
passing from $m$-coordinates to $k$-coordinates
usually simplifies the description of interesting varieties  in~$\PP^N$.

We define the {\em affine Gaussian moment variety}
to be the intersection of  $\mathcal{G}_{n,d} $ with the
the affine chart $\mathbb{A}^{N} = \{ m_{00\cdots 0} = 1\}$ in $\PP^N$.
The transformation (\ref{eq:explog}) between
moments and cumulants is an isomorphism.
Under this isomorphism, the affine Gaussian moment variety
is the linear space defined by the vanishing 
of all cumulants of orders $3,4,\ldots,d$. This implies:

\begin{remark}
The affine moment variety
$\,\mathcal{G}_{n,d} \cap \mathbb{A}^N\,$ is  an affine
space of dimension $n + \binom{n+1}{2}$.
\end{remark}

For instance, the $5$-dimensional
affine variety $\mathcal{G}_{2,3} \cap \mathbb{A}^9$ is 
isomorphic to the $5$-dimensional linear space
 defined by $k_{30} = k_{21} = k_{12} = k_{03} = 0$.
This was translated into moments
in (\ref{eq:saturation}).

For the purpose of studying mixtures, the first truly interesting bivariate
case is $d = 4$. Here the affine moment variety 
$\mathcal{G}_{2,4} \cap \mathbb{A}^{14}$ is defined by the vanishing of
the nine cumulants
$$ \begin{matrix}
k_{03}  = & 2 m_{01}^3-3 m_{01} m_{02}+m_{03} \\
k_{12}  = &  2 m_{01}^2 m_{10}-2 m_{01} m_{11}-m_{02} m_{10}+m_{12} \\
k_{21}  = &  2 m_{01} m_{10}^2-m_{01} m_{20} - 2 m_{10}m_{11}+m_{21} \\
k_{30}  = & 2 m_{10}^3-3 m_{10} m_{20}+m_{30} \\
k_{04}  = &  -6 m_{01}^4+12 m_{01}^2 m_{02}-4 m_{01} m_{03}-3 m_{02}^2+m_{04} \\
k_{13}  = & -6 m_{01}^3 m_{10} + 6 m_{01}^2 m_{11}+6 m_{01} m_{02} m_{10}
-3 m_{01} m_{12} -3 m_{02} m_{11} m_{03} m_{10} +m_{13} \\
k_{22}  = &\!\!\! -6 m_{01}^2 m_{10}^2 {+}2 m_{01}^2 m_{20}
{+}8 m_{01} m_{10} m_{11}{+}2 m_{02} m_{10}^2{-}2 m_{01} m_{21}
{-}m_{02} m_{20} {-} 2 m_{10} m_{12} {-} 2 m_{11}^2{+}m_{22} \\
k_{31}  = & -6 m_{01} m_{10}^3+6 m_{01}m_{10}m_{20}
+6m_{10}^2 m_{11}-m_{01} m_{30}-3 m_{10} m_{21}-3 m_{11} m_{20}+m_{31} \\
k_{40}  = & -6 m_{10}^4+12 m_{10}^2 m_{20}-4 m_{10}m_{30}-3 m_{20}^2+m_{40}
\end{matrix}
$$
The ideal of the projective variety $\mathcal{G}_{2,4}$
is obtained from these nine polynomials by
homogenizing and saturating with 
a new unknown $m_{00}$. The result of this computation is as follows.

\begin{proposition} \label{prop:G24}
The $5$-dimensional variety $\mathcal{G}_{2,4}$ has degree $102$ in $\PP^{14}$.
Its prime ideal is minimally generated by
$99$ cubics, $41$ quartics, and one quintic. The Hilbert series equals
$$ \dfrac{1+9t+45t^2+66t^3-27t^4+13t^5-8t^6+4t^7-t^8}{(1-t)^{6}}. $$
\end{proposition}

We note that the  moment variety $\mathcal{G}_{2,4}$ contains the
{\em quartic Veronese surface} $\nu_4 (\PP^2) $. This surface is defined 
by $75$ binomial quadrics in
$\PP^{14}$, namely the $2 \times 2$-minors of the  matrix
\begin{equation}
\label{eq:sixbysix}
\begin{pmatrix}
m_{00} & m_{01} & m_{02} & m_{10} & m_{11} & m_{20} \\
m_{01} & m_{02} & m_{03} & m_{11} & m_{12} & m_{21} \\
m_{02} & m_{03} & m_{04} & m_{12} & m_{13} & m_{22} \\
m_{10} & m_{11} & m_{12} & m_{20} & m_{21} & m_{30} \\
m_{11} & m_{12} & m_{13} & m_{21} & m_{22} & m_{31} \\
m_{20} & m_{21} & m_{22} & m_{30} & m_{31} & m_{40} 
\end{pmatrix}.
\end{equation}
As observed in \cite[Section 4.3]{CCMRZ},
this is just a linear coordinate space in cumulant coordinates:
$$
\nu_4 (\PP^2) \cap \mathbb{A}^{14}  \, = \,
V(k_{20},k_{11},k_{02},k_{30},k_{21},k_{12},k_{03},k_{40},k_{31},k_{22},k_{13},k_{04})
\, = \, V(k_{20},k_{11},k_{02}) \,\cap \,\mathcal{G}_{2,4}.
$$

The secant variety $\sigma_2(\nu_4(\PP^2))$ comprises all
ternary quartics of tensor rank $\leq 2$. It has dimension
$5$ and degree $75$ in $\PP^{14}$, and its homogeneous prime
ideal is minimally generated by $148$ cubics, namely the
$3 \times 3$-minors of the $6 \times 6$
Hankel matrix in (\ref{eq:sixbysix}). Also this ideal 
becomes much simpler when passing from moments to cumulant coordinates.
Here, the ideal of $\sigma_2(\nu_4(\PP^2)) \cap \mathbb{A}^{14}$
is generated by $36$ binomial quadrics,
like $\,k_{31}^2-k_{22} k_{40}\,$ and $\, k_{30} k_{31}-k_{21} k_{40}$,
along with seven trinomial cubics like
$\,2 k_{20}^3-k_{30}^2+k_{20} k_{40}\,$ and $\,
      2 k_{11} k_{20}^2-k_{21} k_{30}+k_{11} k_{40}$.

\begin{remark} \rm
The Gaussian moment variety $\mathcal{G}_{2,5}$
has dimension $5$ in $\PP^{19}$, and we found its
degree to be $332$. However, at present,
 we do not know a generating set for its prime ideal.
\end{remark}
      
We close this section by reporting the computation of the
first interesting case for $n=3$.

\begin{proposition}
The Gaussian moment variety $\mathcal{G}_{3,3}$ has 
dimension $9$ and degree $130$ in $\PP^{19}$.
Its prime ideal is minimally generated by $84$ cubics, $192$ quartics,
 $21$ quintics, $15$ sextics, $36$ septics, and $35$ octics.
 The Hilbert series equals
 $$
 \frac{ 1{+}10t {+} 55t^2 {+}136t^3 {-}26t^4 {-}150t^5 
 {+}139t^6 {-}127t^7 {+}310t^8 {-}449t^9 {+}360t^{10} {-}160t^{11} 
{+}32t^{12} {-}t^{13}}{(1-t)^{10}}.
$$
 \end{proposition}

\section{Pearson's Crabs: Algebraic Statistics in 1894}
\label{sec:three}

The method of moments in statistics was introduced by
Pearson in his 1894 paper \cite{Pea}. In our view, this
can be regarded as the beginning of Algebraic Statistics.
In this section we revisit Pearson's computation and
related work of Lazard \cite{Laz}, 
and we extend them further.

The first six moments were expressed in (\ref{eq:karlpara}) in
terms of the parameters. The equation 
$K = {\rm log}(M)$ in (\ref{eq:explog}) 
writes the first  six cumulants 
in terms of the first six moments:
 \begin{equation}
\label{eq:karlpara2}
 \begin{matrix}
k_1 & = & m_1 \\
k_2 & = & m_2 - m_1^2  \\
k_3 & = & m_3 - 3 m_1 m_2 + 2 m_1^3 \\
k_4 & = & m_4 - 4 m_1 m_3 - 3 m_2^2 + 12 m_1^2 m_2 - 6 m_1^4 \\
k_5 & = & m_5 - 5 m_1 m_4 -10 m_2 m_3 + 20 m_1^2 m_3 + 30 m_1 m_2^2 - 60 m_1^3 m_2 + 24 m_1^5 \\
k_6 & = & m_6 - 6m_1 m_5 - 15 m_2 m_4 + 30 m_1^2 m_4 -10 m_3^2 + 120 m_1 m_2 m_3 - 120 m_1^3 m_3 \\
    &  & +30m_2^3 -270 m_1^2 m_2^2 + 360 m_1^4 m_2 - 120 m_1^6 
\end{matrix}
\end{equation}

Pearson's method of moments 
 identifies the parameters in a mixture of two 
univariate Gaussians. Suppose 
the first five moments $m_1,m_2,m_3,m_4,m_5$
are given numerically from data.
Then we obtain numerical values for $k_1,k_2,k_3,k_4,k_5$
from the formulas in (\ref{eq:karlpara2}).
Pearson \cite{Pea} solves the corresponding
five equations in  (\ref{eq:karlpara}) for the
five unknowns $\lambda,\mu,\nu,\sigma,\tau$.
The crucial first step is to find the roots of the
following univariate polynomial of degree $9$ in $p$.

\begin{proposition}
The product of normalized means $p = (\mu - m_1) (\nu - m_1)$ 
satisfies 
\begin{equation}
\label{pearsonpoly}
\begin{matrix}
8p^9+28k_4p^7+12k_3^2p^6+(24k_3k_5+30k_4^2)p^5+(148k_3^2k_4-6k_5^2)p^4 \\
+(96k_3^4+9k_4^3
-36k_3k_4k_5)p^3+(-21k_3^2k_4^2 -24k_3^3k_5)p^2 - 32k_3^4k_4p-8k_3^6 \,\, = \,\, 0.
\end{matrix} 
 \end{equation}
\end{proposition}

\begin{proof}
We first prove the identity (\ref{pearsonpoly})
under the assumption that the empirical mean is zero:
\begin{equation}
\label{eq:zeromean} m_1\, =\,   \lambda \mu + (1-\lambda) \nu   \,\,= \,\, 0 . 
\end{equation}
In order to work modulo the symmetry that switches the two 
Gaussian components, we replace the
unknown means $\mu$ and $\nu$ by their 
first two elementary symmetric polynomials:
\begin{equation}
\label{eq:karlpara4}
p= \mu \nu \quad  \hbox{and} \quad  s = \mu + \nu . 
\end{equation}
In \cite{Pea}, Pearson applies considerable effort and cleverness  to eliminating
the unknowns $\mu, \nu, \sigma, \tau, \lambda $ from the constraints
(\ref{eq:karlpara}), (\ref{eq:karlpara2}), (\ref{eq:karlpara4}). 
We here offer a derivation that can be checked
easily in a computer algebra system.
We start by solving
(\ref{eq:zeromean}) for $\lambda$. Substituting
\begin{equation}
\label{eq:getlambda}
\lambda \,=\, \frac{-\nu}{\mu - \nu}.
\end{equation} 
 into $k_2 = \lambda (\mu^2 + \sigma^2) + (1-\lambda) (\nu^2 + \tau^2) $,
 we obtain the relation $k_2 = -R_1 - p$, where  
\begin{equation}
\label{eq:getR1}
R_1 \,=\,\frac{\sigma^2 \nu - \tau^2 \mu}{\mu - \nu}.
\end{equation}
This the first of a series of semi-invariants $R_i$ that appear naturally 
when trying to write the cumulant expressions in terms of $p$ and $s$.
In the next instance, by letting 
\begin{equation}
\label{eq:getR2}
R_2\,=\,\frac{\sigma^2 - \tau^2}{\mu - \nu}
\end{equation}
we can write $k_3 = -(3R_2 + s)p$.
In a similar way, we obtain
\begin{equation}
\label{eq:cumR}
 \begin{matrix}
k_4 & =& 3R_3 + p(p-s^2) - 3k_2^2  \\
k_5 & =& 5R_4p - sp(s^2-2p) - 10k_2 k_3 \\
k_6 & =& 15R_5 - p(s^4-3s^2p +p^2) - 15k_2^3
 -15k_2 k_4 - 10k_3^2 \end{matrix}
\end{equation}
where 
\begin{equation}
 \begin{matrix}
R_3 & =& (\mu \sigma^4 - \nu \tau^4 + 2\mu \nu^2 \tau^2 - 2\mu^2 \nu \sigma^2)/(\mu - \nu)  \\
R_4 & =& (3\tau^4 - 3 \sigma^4 + 2\nu^2 \tau^2 - 2 \mu^2 \sigma^2)/(\mu - \nu) \\
R_5 & =& (\mu^4 \nu \sigma^2 - \mu \nu^4 \tau^2 + 3\mu^2 \nu \sigma^4 - 3\mu \nu^2 \tau^4 + \nu \sigma^6 - \mu \tau^6)/(\mu - \nu). \end{matrix}
\end{equation}
It turns out that $R_3,R_4,R_5$ are not independent of $R_1,R_2$.
Namely, we find
\begin{equation}
\label{eq:R3R4R5}
 \begin{matrix}
R_3 & =& R_1^2 + 2pR_1 - 2spR_2 - pR_2^2  \\
R_4 & =& 2sR_1 +6R_1 R_2 + 2(p-s^2)R_2 -3sR_2^2 \\
R_5 & =& -R_1^3 - 3pR_1^2 + (s^2p - p^2)R_1 + 6spR_1 R_2 + 3pR_1 R_2^2 \\
    &   & +(2sp^2 - s^3p)R_2 + (3p^2-3s^2p)R_2^2 - spR_2^2. \end{matrix}
\end{equation}
We now express the three right hand sides
in terms of $p, s, k_2,k_3$ using the relations
\begin{equation}
\label{eq:R1R2}
R_1  \,=\, -k_2 - p \qquad \hbox{and} \qquad
R_2  \,=\, -\dfrac{s}{3} - \dfrac{k_3}{3p}.
\end{equation}
Plugging the resulting expressions for $R_3$ and $R_4$
 into the first two equations of (\ref{eq:cumR}), we get
 \begin{equation}
 \begin{matrix} \label{twoeqns}
-2p^2s^2 - 4spk_3 +6p^3 + 3k_4 p +k_3^2  & =& 0, \\
-2p^2s^3 + 4p^3s + 5sk_3^2 - 20p^2k_3 + 3k_5p & =& 0. \\
\end{matrix}
\end{equation}
Pearson's polynomial (\ref{pearsonpoly}) is 
the resultant of these two polynomials with respect to $s$.

The proof is completed by noting that the entire derivation
is invariant under replacing the parameters for the means $\mu$ and $\nu$
by the normalized means $\nu-m_1$ and $\nu-m_2$.
\end{proof}

\begin{remark} \rm
Gr\"obner bases reveal the following
consequence of the two equations in (\ref{twoeqns}):
\begin{equation}
\label{eq:finds}
 (4 p^3 k_3-4 k_3^3 - 6 p k_3 k_4-2 p^2 k_5) s 
+4 p^5+14 p^2 k_3^2+8 p^3 k_ 4+k_3^2 k_4+3 p k_4^2-2 p k_3 k_5 \,\,=\,\,0.
\end{equation}
This furnishes an 
expression for $s$ as rational function
in the quantities $k_3,k_4,p$.
Note that (\ref{pearsonpoly}) and (\ref{eq:finds}) do not
 depend on $k_2$ at all. The second moment $m_2$ is only involved via $k_4$.
\end{remark}

We next derive the equation of 
the secant variety that was promised in the Introduction.
 
\begin{proof}[Proof of Theorem \ref{thm:thirtynine}]
Using (\ref{eq:R3R4R5}) and (\ref{eq:R1R2}),  the last equation in (\ref{eq:cumR}) translates into
 \begin{equation} 
 \label{thirdeqn}
 \begin{matrix}
-144p^5+(72s^2-270k_2)p^4+(90s^2k_2 +180sk_3 - 4s^4)p^3 \,+ \\ (-135k_2k_4
  + 180sk_2k_3 - 30s^3 k_3 
 - 90k_3^2- 9k_6)p^2 - 30k_3^2 (s^2 + \frac{3}{2}k_2)p + 5sk_3^3 & = & 0.
\end{matrix}
\end{equation}
We now eliminate the unknowns $p$ and $s$ from the three equations
 in (\ref{twoeqns}) and (\ref{thirdeqn}).
After removing an extraneous factor $k_3^3$, we obtain
an irreducible polynomial in $k_3,k_4,k_5,k_6$
of degree $23$ with $195$ terms.
This polynomial is also mentioned in \cite[Proposition 12]{Laz}.

We finally substitute the expressions in (\ref{eq:karlpara2})
to get an inhomogeneous polynomial
in $m_1,m_2,\ldots,m_6$ of degree $39$ with $31154$ terms.
At this point, we
check that this polynomial vanishes at the parametrization (\ref{eq:karlpara}).
 To pass from affine space $\mathbb{A}^6$ to projective space $\PP^6$,
we  introduce the homogenizing variable $m_0$,
by replacing $m_i$ with $m_i/m_0$ for $i=1,2,3,4,5,6$ and clearing denominators.
The degree in each moment $m_i$ is  read off by inspection.
\end{proof}

\begin{remark} \rm
The elimination in the proof above can be carried out by
computing a Gr\"obner basis for the ideal that is obtained by
adding (\ref{thirdeqn}) to the system  (\ref{twoeqns}).
Such a Gr\"obner basis reveals that both $p$ and $s$
can be expressed as rational functions in the cumulants.
This confirms Lazard's result \cite{Laz} that
Gaussian mixtures for $k=2$ and $n=1$ are  rationally identifiable
from their moments up to order six.  We stress that
Lazard \cite{Laz} does much more
than proving rational identifiability: he also provides a very detailed analysis 
of the real structure and special fibers of the map
$(\lambda, \mu,\nu,\sigma,\tau) \mapsto
(m_1,m_2,m_3,m_4,m_5,m_6)$ in (\ref{eq:karlpara}).
\end{remark}

We close this section by stating the
classical method of moments 
and by revisiting Pearson's 
application to crab measurements.
For $k=2,n=1$, the method
works as follows. From the data, compute the
empirical moments $m_1,\ldots,m_5$, and 
derive the cumulants $k_3,k_4,k_5$ via (\ref{eq:karlpara2}).
Next compute the nine complex zeros of the 
Pearson polynomial (\ref{pearsonpoly}).
We are only interested in zeros $p$ that are real and 
non-positive, because $\, (\mu - m_1) (\nu - m_1) \leq 0$.
All other zeros get discarded. For each non-positive zero $p$
of (\ref{pearsonpoly}), compute the corresponding $s$ from (\ref{eq:finds}).
By (\ref{eq:karlpara4}), we obtain $\mu $ and $\nu$ as the two zeros of the equation
$\, x^2 - s x + p = 0$.
The mixture parameter $\lambda$ is given by (\ref{eq:getlambda}).
Finally,  since $R_1$ and $R_2$ are now known by (\ref{eq:R1R2}),
we obtain $\sigma^2$ and $\tau^2$ by solving an inhomogeneous system 
of two linear equations, (\ref{eq:getR1}) and (\ref{eq:getR2}).

The  algorithm in the previous paragraph works
well when $m_1,m_2,m_3,m_4,m_5$ are general enough.
For special values of the empirical moments, however,
one might encounter zero denominators and other degeneracies.
Extra care is needed in those cases. We implemented
a complete method of moments (for $n=1,k=2$) in
the statistics software {\tt R}. Note that what we described above
computes $\mu-m_1, \nu-m_1$, so we should
add $m_1$ to recover $\mu$ and $\nu$.

Pearson \cite{Pea} applied his method to measurements taken from crabs in the Bay of
Naples, which form  different populations.
His data set is the histogram shown in blue in Figure \ref{fig:eins}.

\begin{figure}[ht!]
\centering
\includegraphics[width=100mm]{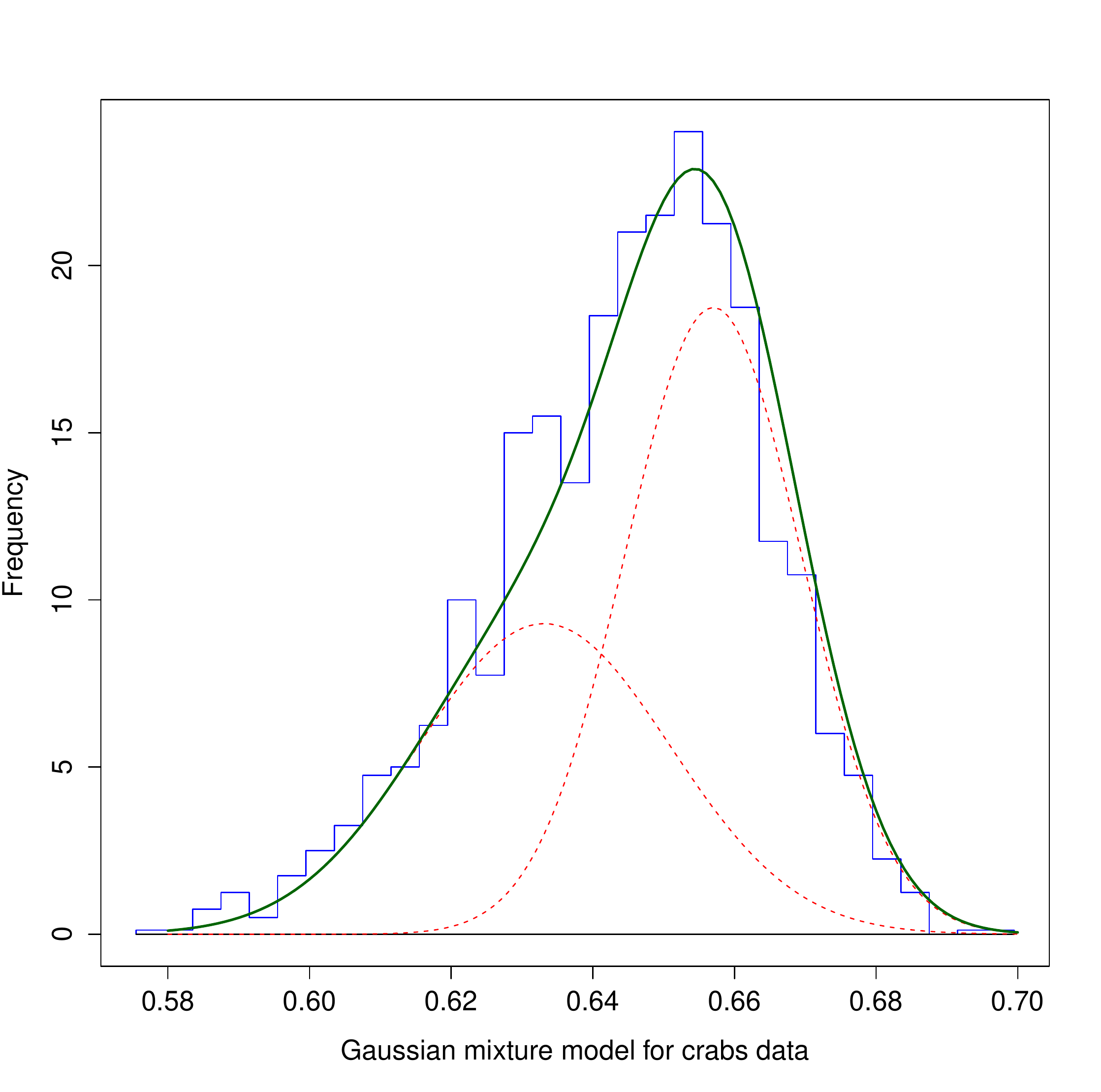}
\caption{\label{fig:eins} 
The crab data in the histogram is approximated by
the mixture of two Gaussians. Pearson's method gives the parameters
 $\mu=0.633, \sigma=0.018, \nu=0.657, \tau=0.012, \lambda=0.414$}
\end{figure} 

Pearson computes the empirical moments from the crab data, and he
takes these as the numerical values for $m_1,m_2,m_3,m_4,m_5$.
The resulting nonic polynomial (\ref{pearsonpoly}) has three real roots,
two of which are non-positive. One computes the model parameters
as above. At this point, Pearson has 
     two statistically meaningful solutions. To choose between them, he  
     computes $m_6$ in each case, and selects the model that is closest to the
     empirical $m_6$. The resulting probability density function
     and  its mixture components are shown in  Figure~\ref{fig:eins}.

\section{Mixtures of Univariate Gaussians}
\label{sec:four}

Our problem is to study the higher secant variety $\sigma_k (\mathcal{G}_{1,d})$
of the moment surface $\mathcal{G}_{1,d} \subset \PP^d$ whose equations were given in
Proposition  \ref{prop:surface}. The hypersurface $\sigma_2 (\mathcal{G}_{1,6})$
was treated in Theorem \ref{thm:thirtynine}. In the derivation of its equation 
in the previous section, we started out with introducing the new unknowns
 $s= \mu + \nu$ and $p=\mu \nu$. 
After introducing cumulant coordinates, the defining expressions for
the moments $m_4, m_5, m_6$ in (\ref{eq:karlpara}) turned into the
 three equations (\ref{twoeqns}),(\ref{thirdeqn}) in $k_2,k_3,k_4,k_5,k_6,s,p$,
 and from these we then eliminated $s$ and $p$.

The implicitization problem for  $\sigma_2 (\mathcal{G}_{1,d})$ when $d>6$ can be
approached with the same process. Starting from the moments, we derive
polynomials in $ k_2,k_3,\ldots,k_d,s,p$ that contain $k_d$ linearly.
The extra polynomial that contains $k_7$ linearly and
is used for $\sigma_2 (\mathcal{G}_{1,7})$ equals
\begin{equation}
\label{eq:for17}
\begin{matrix}
16 p^3 s^5-126 k_2 p^3 s^3+42 k_3 p^2 s^4-148 p^4 s^3+252 k_2 p^4 s
-126 k_3 p^3 s^2 \\ +216 p^5 s  +315 k_2 k_3^2 p s -1260 k_2 k_3 p^3-35 k_3^3 s^2
+210 k_3^2 p^2 s-378 k_3 p^4 \\ +189 k_2 k_5 p^2+35 k_3^3 p+315 k_3 k_4 p^2+9 k_7 p^2.
\end{matrix}
\end{equation}
The extra polynomial that contains $k_8$ linearly and is used for $\sigma_2 (\mathcal{G}_{1,8})$ equals
\begin{equation}
\label{eq:for18}
\begin{matrix}
20 p^4 s^6+336 k_2 p^4 s^4-112 k_3 p^3 s^5+124 p^5 s^4-3780 k_2^2 p^4 s^2
+2520 k_2 k_3 p^3 s^3-6048 k_2 p^5 s^2 \\ - 420 k_3^2 p^2 s^4+2128 k_3 p^4 s^3  
-2232 p^6 s^2-7560 k_2^2 k_3 p^3 s+11340 k_2^2 p^5+2520 k_2 k_3^2 p^2 s^2 
\\ -15120 k_2 k_3 p^4 s +12096 k_2 p^6  -280 k_3^3 p s^3 +2940 k_3^2 p^3 s^2-7056 k_3 p^5 s
+3564 p^7  \\
+1890 k_2^2 k_3^2 p^2  +5670 k_2^2 k_4 p^3 -420 k_2 k_3^3 p s
 +7560 k_2 k_3^2 p^3+35 k_3^4 s^2
+280 k_3^3 p^2 s \\ -1260 k_3^2 p^4  +756 k_2 k_6 p^3-35 k_3^4 p+1512 k_3 k_5 p^3
+945 k_4^2 p^3+27 k_8 p^3.
\end{matrix}
\end{equation}

\begin{proposition}
The ideals of the $5$-dimensional varieties 
$\sigma_2 (\mathcal{G}_{1,7}) \cap \mathbb{A}^{7}$ and 
$\sigma_2 (\mathcal{G}_{1,8}) \cap \mathbb{A}^{8}$
in cumulant coordinates are obtained from (\ref{twoeqns}), (\ref{thirdeqn}),
(\ref{eq:for17}) and (\ref{eq:for18}) by eliminating $s$ and $p$.
\end{proposition}

The polynomials above represent a sequence of birational maps
$\,\sigma_2 (\mathcal{G}_{1,d}) \dashrightarrow \sigma_2 (\mathcal{G}_{1,d-1})$,
which allow us to recover all cumulants from earlier cumulants and the parameters $p$
and $s$. In particular, by solving the equation (\ref{pearsonpoly}) for $p$
and then recovering $s$ from (\ref{eq:finds}), we can invert the
parametrization for any of the moment varieties $\sigma_2(\mathcal{G}_{1,d})
\subset \PP^d$. If we are given $m_1,m_2,m_3,m_4,m_5$ from data  then 
we expect $18 = 9 \times 2$ complex solutions $(\lambda,\mu,\nu,\sigma,\tau)$.
The extra factor of $2$ comes from label swapping between the two Gaussians. 
In that sense, the number $9$ is the algebraic degree of
the identifiability problem  for $n=1$ and $k=2$.

We next move on to $k=3$. There are now eight model parameters.
These are mapped to $\PP^8$ with coordinates $(m_0:m_1:\cdots:m_8)$,
and we are interested in the degree of that map.

Working in cumulant coordinates
as in Section \ref{sec:three}, and
using the Gr\"obner basis package {\tt FGb} in {\tt maple},
we computed the degree of that map. It turned out to be
$ 1350  = 3! \cdot 225$.

\begin{theorem}
The mixture model of $\,k=3$ univariate Gaussians is algebraically identifiable
from its first eight moments.
The algebraic degree of this identifiability problem
equals $225$.
\end{theorem}

We also computed a generalized Pearson polynomial of degree $225$ for $k=3$.
Namely, we replace the three means $\mu_1,\mu_2,\mu_3$ by their elementary symmetric
polynomials  $e_1=\mu_1 + \mu_2 + \mu_3$, $e_2 = \mu_1 \mu_2 + \mu_1 \mu_3 + \mu_2 \mu_3$ 
and $e_3 = \mu_1 \mu_2 \mu_3$. This is done by a derivation analogous
to (\ref{eq:getR2})-(\ref{twoeqns}). This allows us to eliminate
all model parameters other than $e_1,e_2,e_3$.
The details and further computational results will be presented in a forthcoming article.

  We compute a lexicographic
Gr\"obner basis $\mathcal{G}$ for the above equations in $\mathbb{R}[e_1,e_2,e_3]$,
with generic numerical values of the eight moments $m_1,\ldots,m_8$.
 It has  the expected shape
$$  \mathcal{G} \,=\,
\bigl\{ f(e_1), e_2-g(e_1), e_3 - h(e_1) \bigr\}. $$
Here $f,g,h$ are univariate polynomials of degrees 
$225,224,224$ respectively. In particular,
$f$ is the promised generalized Pearson polynomial of degree $225$
for mixtures of three Gaussians.

For general $k$, the mixture model has $3k-1$ parameters. 
Based on what we know for $k=2$ and $k=3$, we offer the following conjecture
concerning the identifiability of Gaussian mixtures. Recall that the
{\em double-factorial} is the product of the smallest odd positive integers:
$$ (2k-1) !! \,\,\, = \,\,\, 1 \cdot 3 \cdot 5 \cdot \cdots \, (2k-1) . $$

\begin{conjecture} \label{conjidentif}
The mixture model of $k$ univariate Gaussians is algebraically identifiable 
by the moments of order $\leq 3k-1$, and the degree of this
identifiability problem equals $\bigl( \,(2k-1)!! \,\bigr)^2$. 
Moreover, this model is rationally identifiable by the moments of order $\leq 3k$.
\end{conjecture}

Geometrically, this conjecture states that
the moment variety  $\sigma_k (\mathcal{G}_{1,3k-1}) $ fills the
ambient space $ \PP^{3k-1}$, and that  $\sigma_k (\mathcal{G}_{1,3k}) $ is
a hypersurface in $ \PP^{3k}$ whose secant parametrization is birational.
As explained in the Introduction, a priori we only know that the dimension of 
$\sigma_k (\mathcal{G}_{1,d}) $ is equal to $3k-1$ for $ d \gg 0$. What Conjecture 
\ref{conjidentif} implies is that this already holds for $d=3k-1$, so that the secant 
varieties always have the expected dimension. We know 
this result for $k=2$ by the work of Pearson \cite{Pea} and Lazard \cite{Laz}
that was discussed in Section \ref{sec:three}. For $k=3$ we verified the first part
of the conjecture, but we do not yet know whether rational identifiability holds.
Also, we do not know the degree of the hypersurface
$\sigma_3( \mathcal{G}_{1,9}) \subset \PP^9$.
The double-factorial conjecture for the degree is nothing but a wild guess. 

Computations for $k=4$ appear currently out of reach for Gr\"obner basis methods.
If our wild guess is true then 
the expected number of complex solutions for the $11$ moment equations 
whose solution identifies a mixture of $k=4$ univariate Gaussians is $105^2 \times 4! = 264,600$.

\section{Method of Moments versus Maximum Likelihood}
\label{sec:five}

In \cite[Section 3]{CB}, the sample consisting of the following $N=2K$ data points was examined:
\begin{equation}
\label{eq:specialdata}
 1,\,1.2,\,2,\,2.2,\,3,\,3.2,\,4,\ldots,K,\,K+0.2 \qquad \hbox{(for $K>1$).} 
 \end{equation}
 Its main purpose was to show that, unlike most models studied in Algebraic Statistics, there is no notion of maximum likelihood degree (or {\em ML degree}; see \cite{DSS}) for a mixture of two Gaussians. 
 Indeed, the particular sample in (\ref{eq:specialdata})
 has the property that, as $K$ increases, the number of critical points of the log-likelihood function grows without any bound. More precisely, for each `cluster' or pair $(k,k+0.2)$, one can find a non-trivial critical point $(\hat{\lambda},\hat{\mu},\hat{\nu},\hat{\sigma},\hat{\tau})$
 of the likelihood equations such that the mean estimate $\hat{\mu}$ lies between them.

In this section we apply Pearson's method of moments to this sample. 
The special nature of the data raises some interesting considerations.
As we shall see, the even
spacing of the points in the list (\ref{eq:specialdata})
 implies that all empirical cumulants of odd order $\geq 3 $ vanish:
\begin{equation}
\label{eq:oddcumulants}
 k_3 \,=\, k_5 \,=\, k_7 \,=\, k_9 \,=\, \cdots  \,=\, 0 .
\end{equation}
Let us analyze what happens when applying the method of moments to \textit{any} sample 
that satisfies (\ref{eq:oddcumulants}).
Under this hypothesis Pearson's polynomial (\ref{pearsonpoly}) factors
as follows:
\begin{equation}
\label{eq:smallpearson}
8p^9 + 28p^7k_4 + 30p^5k_4^2 + 9p^3k_4^3 \,\,  = \,\,
8p^3\left(p^2 + \frac{3}{2}k_4 \right)^2 \left( p^2+\frac{1}{2}k_4 \right)
\,\,=\,\,0.
\end{equation}
Recall that $p$ represents $p=(\mu - m_1) (\nu - m_1)$.
The first root of the Pearson polynomial is $p = 0$.
This implies $m_1 = \mu$ or $m_1 = \nu$. 
Since $m_1$ is the mean of $\mu$ and $\nu$, we conclude
that the means are equal: $m_1 = \mu = \nu$.  However, the 
equal-means model cannot be recovered from the first five moments.
To see this, note that the equations for cumulants $k_1=0$, $k_3=0$ and $k_5=0$ become $0=0$, yielding no information on the remaining three parameters.

If we assume that also the sixth moment $m_6$ is known from the data, then 
the parameters can be identified.  The original system (\ref{eq:karlpara})
 under the equal-means model $\mu=\nu=0$ equals
\begin{equation}
 \begin{matrix}
m_2 & = & \lambda \sigma^2 + (1-\lambda) \tau^2 \\
m_4 & = & 3 \lambda  \sigma^4
          + 3 (1-\lambda)  \tau^4 \\
m_6 & = & 15 \lambda  \sigma^6 
          + 15 (1-\lambda) \tau^6.
\end{matrix}
\end{equation}
After some rewriting and elimination:
 \begin{equation} \label{eqmeansol}
 \begin{matrix}
\lambda (\sigma^2 - \tau^2) & = & k_2 - \tau^2 \\
5k_4 (\sigma^2 + \tau^2) & = & 10k_2 k_4 + k_6 \\
15k_4 (\sigma^2 \tau^2) & = & 3k_2 k_6 + 15k_2^2 k_4 - 5k_4^2.
\end{matrix}
\end{equation}
Assuming $k_4 \not = 0$, this system can be solved easily in radicals 
for $\lambda, \sigma, \tau$.

If $k_4 \geq 0$ then $p=0$ is the only real zero of (\ref{eq:smallpearson}).
If $k_4 < 0$ then two other solutions are:
\begin{equation}
\label{eq:zweip}
p\,=\,-\sqrt{\frac{-3}{2}k_4} \quad \hbox{and} \quad p\,=\,-\sqrt{\frac{-1}{2}k_4}.
\end{equation}  
Note that $p$ must be negative because it is the product
of the two normalized means.

The mean of the sample in (\ref{eq:specialdata}) is $\,m_1 \,= K/2+3/5$.
The central moments are
\begin{equation}
\label{eq:samplemoments}
 m_r \quad = \quad \frac{1}{2K} \cdot \biggl( 
\sum_{i=1}^K \bigl(i-m_1 \bigr)^r \, + \,
\sum_{i=1}^K \bigl(i-m_1 + \frac{1}{5} \bigr)^r \biggr) 
\qquad \hbox{for} \,\, r = 2,3,4,\ldots. 
\end{equation}
This expression is a polynomial of degree $r$ in $K$.
That polynomial is zero when $r$ is odd.
Using (\ref{eq:karlpara2}), this implies 
the vanishing of the odd sample cumulants
(\ref{eq:oddcumulants}). For even $r$, we get
$$ \begin{matrix}
m_2  = \frac{1}{12}K^2-\frac{11}{150}, \,\,\,\,
m_4  =  \frac{1}{80}K^4 - \frac{11}{300}K^2 + \frac{91}{3750},\,\,\,\,
m_6 =  \frac{1}{448}K^6 - \frac{11}{800}K^4
+ \frac{91}{3000}K^2- \frac{12347}{656250}.
\end{matrix}
$$
These polynomials simplify to binomials when we substitute
the moments into (\ref{eq:karlpara2}):
\begin{equation}
\label{eq:evencumulants}
k_1 = m_1 =  \frac{K}{2}+0.6, \hspace{0.5 cm} k_2 = \frac{K^2}{12}-\frac{11}{150}, \hspace{0.5 cm} 
k_4 = -\frac{K^4}{120} + \frac{61}{7500}, \hspace{0.5 cm} k_6 = \frac{K^6}{252} - \frac{7781}{1968750}.
\end{equation}
These are the sample cumulants. We are aware that
these are biased estimators, and  {\em k-statistics} might be preferable.
However,  for simplicity, we shall use (\ref{eq:evencumulants}) in our derivation.

Since $K \geq 1$, we have $k_4<0$ in (\ref{eq:evencumulants}).
Hence the Pearson polynomial has three distinct real roots.
For $p=0$, substituting (\ref{eq:oddcumulants}) and
(\ref{eq:evencumulants}) into (\ref{eqmeansol}) shows that, for every value of 
$K$, there are no positive real solutions for both $\sigma$ and $\tau$. Thus the method of moments concludes that the sample does \textit{not} come from a mixture of two Gaussians with the same mean.

Next we consider the two other roots in (\ref{eq:zweip}).
To recover the corresponding $s$-values, we use
 the system (\ref{twoeqns}) with all odd cumulants replaced by zero:
\begin{equation}\label{eq:karlpara3}
 \begin{matrix}
p(6p^2 - 2s^2p + 3k_4) & =& 0 \\
2sp^2(2p-s^2) & =& 0 \\
\end{matrix}
\end{equation}
For $p=-\sqrt{\frac{-3}{2}k_4}$, the first equation gives $s \not= 0$,
and  the second yields a non-real value for $s$, so this is not viable.
For $ p=-\sqrt{\frac{-1}{2}k_4}$, we obtain $s=0$,
and this is now a valid solution.

In conclusion, Pearson's method of moments infers a
 non-equal-means model for the data  (\ref{eq:specialdata}).
  Using central moments, i.e.~after
subtracting $m_1 = K/2 + 3/5$ from each data point, we find
 $\mu=-\nu = \sqrt[4]{\frac{-k_4}{2}}$.
   These values lead  to $\lambda=\frac{1}{2}$ and $\sigma = \tau$. The final estimate is
\begin{equation} \label{estimatorMOM}
(\lambda,\mu,\sigma^2,\nu,\tau^2) \,\,=\,\, \left( \frac{1}{2}\,,\,m_1 - \sqrt[4]{\frac{-k_4}{2}}
\,,\, k_2-\sqrt{\frac{-k_4}{2}}\,,\, m_1 + \sqrt[4]{\frac{-k_4}{2}}\,,\,k_2-\sqrt{\frac{-k_4}{2}} \,\right). 
\end{equation}

We are now in a position to compare this estimate
to those found by maximum likelihood.

\begin{example}(Example 2 of \cite{CB} with $K=7$)
\label{ex:K7} \rm
The sample consists of the $14$ data points 1,1.2,2,2.2,3,3.2,4,4.2,5,5.2,6,6.2,7,7.2. 
The method of moments estimator  (\ref{estimatorMOM}) evaluates to
$$ (\lambda,\mu,\sigma,\nu,\tau) \,=\, \left(\, \frac{1}{2}\,,\,\frac{41-\sqrt[4]{100001}}{10}, \frac{\sqrt{401-\sqrt{100001}}}{10}, \frac{41+\sqrt[4]{100001}}{10}, \frac{\sqrt{401-\sqrt{100001}}}{10} \right). $$

For general $k_3,k_4,k_5$, Pearson's equation (\ref{pearsonpoly}) of 
degree $9$ cannot be solved in radicals,
as its roots are algebraic numbers with Galois group $S_9$ over $\mathbb{Q}$.
However, for our special data, the algebraic degree 
of the solution drops, and we could write the estimate in radicals.

The situation is dramatically different for likelihood inference.
It was shown in \cite{CB} that the critical points
for the likelihood function of the mixture of two Gaussians with
data (\ref{eq:specialdata}) have transcendental coordinates,
and that the number of these critical points grows with $K$.

It is thus interesting to assess the quality of our solution 
(\ref{estimatorMOM}) from the likelihood perspective.
The probability density function for the 
Gaussian mixture with these parameters is
shown in Figure \ref{fig:zwei}.
The corresponding value of the log-likelihood function is $ -28.79618895$. 

If the estimate (\ref{estimatorMOM})
 is used as starting point in the EM algorithm, then 
 it converges to the stationary point $(\lambda,\mu,\sigma,\nu,\tau) = (0.500000,2.420362,5.77968,1.090329,1.090329) $.
 That point has  a   log-likelihood value of approximately 
  $-28.43415$. Comparing to Table 1 of \cite{CB}, this value is only beaten by the critical points associated to the endpoints $k=1$ and $k=7$. 

\begin{figure}[ht!]
\centering
\includegraphics[width=100mm]{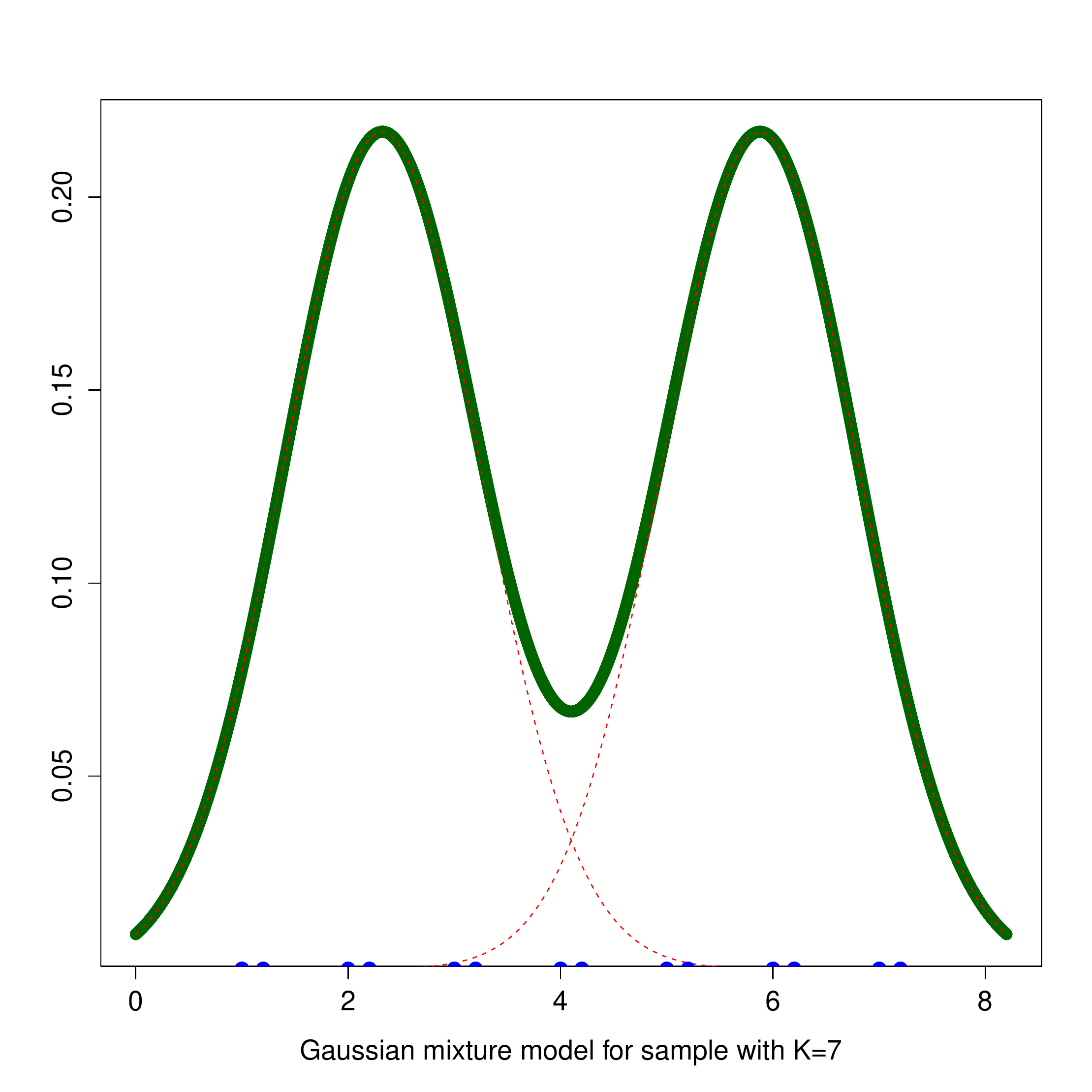}
\caption{\label{fig:zwei} 
The sample data for $K=7$ (in blue) is approximated by
a mixture of two Gaussians via the method of moments.
The parameter values are derived in Example \ref{ex:K7}.
}
\end{figure} 
\end{example}
We make the following observation:  of all the critical points listed in  \cite[Table 1]{CB},
the middle clusters get the lowest log-likelihood.
Hence an equal-means model is not very likely for this sample. This is further confirmed by 
the method of moments (MOM)
since, as mentioned above, the equal-means model is inconsistent with 
our polynomial  equations. 

Behavior similar to Example~\ref{ex:K7} is observed for all $K \geq 2$.
The MOM estimate separates the sample into two halves, and
assigns the same variance to both Gaussian components.
The exact parameter estimates are obtained by
substituting $m_1,k_2,k_4$ from (\ref{eq:evencumulants}) into (\ref{estimatorMOM}).
For $K=20$, the estimate computed by the EM algorithm with the MOM estimate 
as starting point beats in likelihood value all $K$ critical points listed in \cite{CB}.
For $K > 20$,  the likelihood value of the MOM estimate itself appears to be already better than
  the critical points listed in \cite{CB}.    In other words, the
        MOM produces good starting points for maximum likelihood.

\section{Higher Dimensions, Submodels and Next Steps}
\label{sec:six}

At present we know  little about the moment varieties of Gaussian mixtures
for $n \geq 2$, and we see this as an excellent topic for future investigations.
A guiding question is the following:

\begin{problem} \label{prob:needed}
Which order $d$ of cumulants/moments is needed to make the 
mixture model $\sigma_k (\mathcal{G}_{n,d})$ algebraically identifiable?
Which order $d$ is needed to obtain rational identifiability?
\end{problem}

A natural conjecture is that the  dimension of 
the variety $\sigma_k( \mathcal{G}_{n,d})$ 
always coincides with the expected number (\ref{eq:mixturedim}),
unless that number exceeds the dimension $N$ of the ambient projective space.
It is important to note that the analogous statement would not
be true for the submodels where all covariance matrices are zero.
These are the secant varieties of Veronese varieties, and
there is a well-known list of exceptional cases, due to
Alexander and Hirschowitz (cf.~\cite{BrOt}), where
these secant varieties do not have the expected dimension.
However, none of these cases is relevant in the case of
 Gaussian mixtures discussed here.

The following is the first bivariate instance
of the varieties $\sigma_k(\mathcal{G}_{n,d})$ for Gaussian mixtures.

\begin{example} 
\label{ex:s2G24} \rm
Let $k=2, n=2, d = 4$. The variety $\sigma_2( \mathcal{G}_{2,4})$
lives in the $\PP^{14}$ whose coordinates are the moments
up to order $4$.  This is the variety of secant lines
 for the $5$-dimensional variety featured in Proposition \ref{prop:G24}.
We checked that  $\sigma_2( \mathcal{G}_{2,4})$ has
the expected dimension, namely $11$.
\end{example}

We found it difficult to compute polynomials that vanish on
our moment varieties, including $\sigma_2( \mathcal{G}_{2,4})$.
One fruitful direction to make progress would be to first compute
subvarieties  that correspond
to statistically meaningful submodels. Such submodels
arise naturally when the parameters satisfy various natural constraints.
We illustrate this  for a small case.

Fix $k=2,n=2, d=3$. The variety $\sigma_2( \mathcal{G}_{2,3})$ 
is equal to its ambient space $\mathbb{P}^9$.
We consider the two submodels: that given by
equal variances and that given by equal means.
The number of parameters are $8$ and $9$ respectively.
Both of these models are not identifiable.

\begin{proposition}
The equal-means submodel of $\sigma_2( \mathcal{G}_{2,3})$ 
has dimension $5$ and degree $16$. It is identical to the
Gaussian moment variety $\mathcal{G}_{2,3}$ in Proposition \ref{prop:G23} so the mixtures add nothing new in $\mathbb{P}^9$.
The equal-variances submodel of $\sigma_2( \mathcal{G}_{2,3})$
has dimension $7$ and degree $15$ in $\PP^9$. Its ideal 
is Cohen-Macaulay and is generated by the
maximal minors of the
$6 \times 5$-matrix 
\begin{equation}
\label{eq:hbmatrix}
\begin{pmatrix}
   0 &   0 &    m_{00} &   m_{10} &    m_{01} \\
    0 &  m_{10} &    m_{20} &  m_{30} &  m_{21} \\
    m_{01} &    0 &    m_{02} & m_{12} &    m_{03} \\
     0 &  m_{00} &  2 m_{10} &  2 m_{20} &  2 m_{11} \\
     m_{00} &    0 &  2 m_{01} &  2 m_{11} & 2 m_{02} \\
     m_{10} &  m_{01} &  2 m_{11} &  2 m_{21} &  2 m_{12} 
\end{pmatrix}.
\end{equation}
\end{proposition}

This proposition is proved by a direct computation. That the equal-means submodel of $\sigma_2( \mathcal{G}_{2,3})$ equals $\mathcal{G}_{2,3}$ is not so surprising, since the parametrization of the latter is linear in the variance parameters $s_{11},s_{12},s_{22}$.
This holds for all moments up to order 
$3$. The same is no longer true for $d \geq 4$. On the other hand, it was gratifying to see an occurrence,
in the matrix (\ref{eq:hbmatrix}),
of the {\em Hilbert-Burch Theorem}
for Cohen-Macaulay ideals of codimension~$2$.

We already noted that secant varieties of Veronese
varieties arise as the submodels where the variances are zero.
On the other hand, we can also consider the submodels
given by zero means. In that case we get
the secant varieties of varieties of powers
of quadratic forms. The following concrete example
was worked out with some input from Giorgio Ottaviani.

\begin{example} \label{ex:Faa} \rm
Consider the mixture of two bivariate Gaussians that
are centered at the origin. This model has $7$ parameters:
there is one mixture parameter, and 
each Gaussian has a $2 \times 2$ covariance matrix,
with three unknown entries.
We consider the variety $\mathcal{V}$ that is
parametrized by all moments of order exactly $d = 6$.
This variety has only dimension $5$. It lives
in  the $\PP^6$ with coordinates
$m_{06}, m_{15}, \ldots, m_{60}$.  
This hypersurface has degree $15$.
Its points are
the binary octics that are sums of the third powers
of two binary quadrics. Thus, this is the  secant
variety of a linear projection of the third
Veronese surface from $\PP^9$ to $\PP^6$.

The polynomial that defines $\mathcal{V}$ has
$1370$ monomials of degree $15$ in the six unknowns
$m_{06}, m_{15}, \ldots, m_{60}$. In fact, this is the unique
(up to scaling) invariant of binary sextics of degree $15$.
It is denoted $I_{15}$ in Faa di Bruno's book \cite[Table $IV^{10}$]{DB},
where a determinantal formula was given.
A quick way to compute $\mathcal{V}$  by elimination is as follows.
Start with the variety
$\sigma_2 (\nu_3(\PP^2))$
 of symmetric $3 \times 3 \times 3$-tensors of rank $\leq 2$.
This is defined by the maximal minors of a Hankel matrix of size $3 {\times} 6$.
It has degree $15$ 
and dimension $5$ in $\PP^9$. Now project into $\PP^6$. This projection 
has no base points, so the image is a hypersurface of degree~$15$.
\end{example}

In Example \ref{ex:Faa} we fixed the order of the moments.
For certain applications, also taking
moments of two orders  makes sense. For instance,
the tensor power method in machine learning \cite{AG, GHK} uses
the moments of order $d=2$ and $d=3$. It would be interesting
to determine the algebraic relations for these restricted moments.
 Geometrically, we should obtain
interesting varieties, even for  $k = 2$. Here is 
a specific example from machine learning.

\begin{example} \rm
Ge, Huang and Kakade \cite{GHK} focus on 
mixtures of Gaussians with zero mean,
and they show how to identify them numerically
using the moments of order $d=4$ and $d=6$.
We examine the corresponding variety
for $n=k=2$.  This lives in the $\PP^{12}$
with coordinates $m_{00}, 
m_{40}, m_{31}, m_{22}, m_{13}, m_{04},
m_{60},m_{51},m_{42},m_{33},m_{24},m_{15},m_{06}$.
We start with the variety $X$
 that is parametrized by the
$4$th and $6$th powers of binary quadrics.
This variety  has dimension three and degree $27$ in $\PP^{12}$.
We are interested  in the secant variety $\sigma_2(X)$.
This secant variety has the expected dimension $7$, so
the model is algebraically identifiable. We do not know
whether $\sigma_2(X)$ is rationally identifiable.
A relation of lowest degree is the following quartic:
$$
\begin{matrix}
6 m_{15} m_{22} m_{31}^2-10 m_{13} m_{24} m_{31}^2 -2 m_{06} m_{31}^3
+10 m_{04} m_{31}^2 m_{33} -9 m_{15} m_{22}^2 m_{40}+15 m_{13} m_{22} m_{24} m_{40} \\
+2 m_{13} m_{15} m_{31} m_{40} + 3 m_{06} m_{22} m_{31} m_{40}
-5 m_{04} m_{24} m_{31} m_{40}-10 m_{13}^2 m_{33} m_{40}
-m_{06} m_{13} m_{40}^2 \\
+m_{04} m_{15} m_{40}^2 
+10 m_{13}^2 m_{31} m_{42}
-15 m_{04} m_{22} m_{31} m_{42}+5 m_{04} m_{13} m_{40} m_{42} 
-6 m_{13}^2 m_{22} m_{51} \\
+ 9 m_{04} m_{22}^2 m_{51}
-2 m_{04} m_{13} m_{31} m_{51}-m_{04}^2 m_{40} m_{51} + 2 m_{13}^3 m_{60}
-3 m_{04} m_{13} m_{22} m_{60}+m_{04}^2 m_{31} m_{60}
\end{matrix}
$$
\end{example}
 
 \smallskip
 
In summary, the study of moments of mixtures
of Gaussians leads to many interesting projective varieties.
Their geometry is still largely unexplored, and offers a fertile
ground for investigations by algebraic geometers. On the statistical
side, it is most interesting to understand the fibers
of the natural parameterization of the variety
$\sigma_k(\mathcal{G}_{n,d})$.
Problem \ref{prob:needed} serves as the guiding question.
In the case of algebraic identifiability, we are always
interested in finding the algebraic degree of the
parametrization, and in effective methods for solving
for the model parameters.
This will be the topic of a forthcoming paper by the same authors.

\bigskip
\bigskip



\noindent
{\bf Acknowledgments.}
This project was supported by 
the INRIA@Silicon Valley program through the project
GOAL ({\em Geometry and Optimization with Algebraic Methods}),
by a Fellowship from the Einstein Foundation Berlin,
and by the US National Science Foundation.

\bigskip

\begin{small}

\bigskip \bigskip

\footnotesize
\noindent {\bf Authors' addresses:}

\medskip

\noindent 
Carlos Am\'endola, Institute of Mathematics, Technische Universit\"at Berlin, MA 6-2,
10623 Berlin, Germany, {\tt amendola@math.tu-berlin.de}

\medskip

\noindent 
Jean-Charles Faug\`ere, INRIA Paris-Rocquencourt,
 LIP6 - Universit\'e Paris 6,
75005 Paris, France, {\tt Jean-Charles.Faugere@inria.fr}

\medskip

\noindent Bernd Sturmfels, Department of Mathematics, University of\
 California, Berkeley, CA 94720-3840, USA,
{\tt bernd@berkeley.edu}

\end{small}
\end{document}